\documentclass{amsart}

\newtheorem{theorem}{Theorem}[section]
\newtheorem{corollary}[theorem]{Corollary}
\newtheorem{lemma}[theorem]{Lemma}
\newtheorem{proposition}[theorem]{Proposition}

\theoremstyle{definition}
\newtheorem{definition}[theorem]{Definition}

\usepackage{extpfeil} 

\newcommand{\Uq}{U_q(\mathfrak{sl}_2)}
\newcommand{\Unegq}{U_{-q}(\mathfrak{sl}_2)}
\newcommand{\TL}{\mathbf{TL}}
\newcommand{\TLdot}{\mathbf{TL^\bullet}}
\newcommand{\id}{\mathrm{id}}
\begin{document}

\markboth{Stephen J. Bigelow}
{A diagrammatic definition of $U_q(\mathfrak{sl}_2)$}


\title{A DIAGRAMMATIC DEFINITION OF $U_q(\mathfrak{sl}_2)$}

\author{STEPHEN J. BIGELOW}

\address{Department of Mathematics,
University of California at Santa Barbara,
Santa Barbara, California 93106}

\maketitle

\begin{abstract}
We give a diagrammatic definition of $\Uq$
when $q$ is not a root of unity,
including the Hopf algebra structure
and relationship with the Temperley-Lieb category.
\end{abstract}



\section{Introduction}
This paper is about $\Uq$,
one of the simplest examples of a quantum group.
For an account of the early history of quantum groups
and some of their applications,
see \cite{jones}.
Our goal is to give
a definition of $\Uq$ and its representation theory
using formal linear combinations of certain diagrams in the plane.
This diagrammatic approach to algebra
has origins that go back to the use of Feynman diagrams in physics.
For a survey of this and some of its varied applications,
see \cite{BS}.

The Temperley-Lieb category $\TL$ is
a category of certain representations of $\Uq$.
The morphisms are represented
by linear combinations of Temperley-Lieb diagrams.
We will define a category $\TLdot$ that contains $\TL$,
but also allows diagrams with interior endpoints and orientations.

Next we will define a Hopf algebra $H$,
whose diagrams include a vertical pole.
If we work over $\mathbb{C}$ and $q$ is not a root of unity
then we find $\Uq$ as a subalgebra of a quotient of $H$.
The relationship between $\Uq$ and $\TL$
is described by the process of
threading a Temperley-Lieb diagram
in place of the pole in a diagram in $H$.

Orientations appeared in
the earliest applications of Temperley-Lieb diagrams
to ice-type models in statistical mechanics,
such as \cite{lieb}.
The orientations in $\TLdot$ are very similar,
and also satisfy the {\em ice rule},
which says that every crossing has
two arrows pointing in and two pointing out.

Orientations again appeared in work of
Frenkel and Khovanov \cite{FK}.
Their idea is that,
whereas a Temperley-Lieb diagram represents a linear map
between representations of $\Uq$,
an oriented Temperley-Lieb diagram
represents a single matrix entry of that linear map.
Such diagrams form a category
that is basically the same as our $\TLdot$.

Although it does not use the same oriented diagrams,
similar ideas are covered in the ``Kyoto path model'',
pioneered by Kashiwara \cite{kashiwara} and others.

More recently,
Lauda \cite{lauda} used diagrammatic methods
to define a categorified $\Uq$.
It is not clear whether our definitions can also be categorified,
or how they relate to Lauda's.

The main feature of our work that seems to be new is the pole,
which will let us describe both $\Uq$ and its representation theory
in the same picture.
One advantage of this is that the key
``intertwining'' relationship becomes visually obvious.
It is proved in Theorem \ref{thm:intertwine}
by physically sliding one action through the other.
This is reminiscent of Morton's diagrammatic proofs that
certain elements of the Temperley-Lieb algebra commute \cite{morton}.
I hope our approach makes the algebra more accessible to others like me
who are more comfortable with tangle diagrams and skein relations.

Throughout the paper,
we work over a field $\mathbb{F}$ containing an element $q$
that is neither $0$ nor $\pm 1$.
We will also need square roots $\sqrt{q}$ and $\sqrt{-q}$.

\section{The category $\TLdot$}

In this section,
we define a monoidal category $\TLdot$.

We start with a quick review
of the Temperley-Lieb category $\TL$.
The objects are the non-negative integers.
The morphisms from $n$ to $m$
are formal linear combinations
of Temperley-Lieb diagrams
that have $n$ endpoints at the bottom and $m$ at the top.
Composition is by stacking.
A closed loop can be deleted
in exchange for the scalar $q + q^{-1}$.
The tensor product of objects is given by
$n \otimes m = n + m$.
The tensor product $f \otimes g$
of two diagrams $f$ and $g$
is obtained by placing $f$ to the left of $g$.

We also allow diagrams with crossings,
which are defined as follows:
$$
\raisebox{-1.2em}{
    \begin{picture}(3,3)
    \put(0,0){\line(1,1){3}}
    \put(3,0){\line(-1,1){1.2}}
    \put(0,3){\line(1,-1){1.2}}
    \end{picture}
}
= \sqrt{-q}
\raisebox{-1.2em}{
    \begin{picture}(3,3)
    \qbezier(0,0)(1.5,1.5)(0,3)
    \qbezier(3,0)(1.5,1.5)(3,3)
    \end{picture}
}
+ \frac{1}{\sqrt{-q}}
\raisebox{-1.2em}{
    \begin{picture}(3,3)
    \qbezier(0,0)(1.5,1.5)(3,0)
    \qbezier(0,3)(1.5,1.5)(3,3)
    \end{picture}
}.$$
Crossings satisfy Reidemeister moves two and three
(as proved in \cite{kauffman}).

We extend $\TL$ to $\TLdot$ by introducing
diagrams with univalent vertices.
A vertex is the endpoint of a strand,
lying in the interior of the diagram.
We require that,
at every vertex,
the strand must have a horizontal tangent vector,
and must be given an orientation either into or out of the vertex.
Unlike ordinary Temperley-Lieb diagrams,
a diagram with vertices
is not considered up to planar isotopy.
Instead, we only allow planar isotopies
that preserve the horizontal tangent vector at each vertex.
We also impose the following
{\em turning},
{\em confetti},
and {\em cutting} relations.

The {\em turning relations} let us rotate a vertex 
at the expense of a power of $\sqrt{q}$.
\[
\sqrt{q} \raisebox{-0.8em}{
    \begin{picture}(2,2)
    \put(2,1){\vector(-1,0){1}}
    \put(1,1.5){\oval(2,1)[l]}
    \put(1,2){\line(1,0){0.5}}
    \put(1.5,2){\circle*{0.3}}
    \end{picture}
}
= \raisebox{-0.8em}{
    \begin{picture}(2,2)
    \put(2,1){\vector(-1,0){1}}
    \put(2,1){\line(-1,0){2}}
    \put(0,1){\circle*{0.3}}
    \end{picture}
}
= \frac{1}{\sqrt{q}} \raisebox{-0.8em}{
    \begin{picture}(2,2)
    \put(2,1){\vector(-1,0){1}}
    \put(1,0.5){\oval(2,1)[l]}
    \put(1,0){\line(1,0){0.5}}
    \put(1.5,0){\circle*{0.3}}
    \end{picture}
},
\quad
\frac{1}{\sqrt{q}} \raisebox{-0.8em}{
    \begin{picture}(2,2)
    \put(2,1){\line(-1,0){1}}
    \put(1.6,1){\vector(1,0){0}}
    \put(1,1.5){\oval(2,1)[l]}
    \put(1,2){\line(1,0){0.5}}
    \put(1.5,2){\circle*{0.3}}
    \end{picture}
}
= \raisebox{-0.8em}{
    \begin{picture}(2,2)
    \put(2,1){\line(-1,0){1}}
    \put(1.6,1){\vector(1,0){0}}
    \put(2,1){\line(-1,0){2}}
    \put(0,1){\circle*{0.3}}
    \end{picture}
}
= \sqrt{q} \raisebox{-0.8em}{
    \begin{picture}(2,2)
    \put(2,1){\line(-1,0){1}}
    \put(1.6,1){\vector(1,0){0}}
    \put(1,0.5){\oval(2,1)[l]}
    \put(1,0){\line(1,0){0.5}}
    \put(1.5,0){\circle*{0.3}}
    \end{picture}
}.
\]

The {\em confetti relations} let us 
eliminate any straight strand
that has univalent vertices at both ends.
\[
\raisebox{-0.3em}{
    \begin{picture}(3,1)
    \put(0,0.5){\circle*{0.3}}
    \put(3,0.5){\circle*{0.3}}
    \put(0,0.5){\line(1,0){3}}
    \put(0,0.5){\vector(1,0){1.0}}
    \put(3,0.5){\vector(-1,0){1.0}}
    \end{picture}
}
=
\raisebox{-0.3em}{
    \begin{picture}(3,1)
    \put(0,0.5){\circle*{0.3}}
    \put(3,0.5){\circle*{0.3}}
    \put(0,0.5){\line(1,0){3}}
    \put(1.5,0.5){\vector(1,0){1.0}}
    \put(1.5,0.5){\vector(-1,0){1.0}}
    \end{picture}
}
= 0,
\quad
\raisebox{-0.3em}{
    \begin{picture}(2,1)
    \put(0,0.5){\circle*{0.3}}
    \put(2,0.5){\circle*{0.3}}
    \put(0,0.5){\line(1,0){2}}
    \put(0,0.5){\vector(1,0){1.0}}
    \end{picture}
}
= 1.
\]

The {\em cutting relation} lets us 
replace a strand with a sum of ``cut'' strands
with the two possible orientations.
\[
\raisebox{-0.3em}{
    \begin{picture}(3,1)
    \put(0,0.5){\line(1,0){3}}
    \end{picture}
}
=
\raisebox{-0.3em}{
    \begin{picture}(4,1)
    \put(0,0.5){\line(1,0){1.5}}
    \put(0,0.5){\vector(1,0){1}}
    \put(1.5,0.5){\circle*{0.3}}
    \put(2.5,0.5){\circle*{0.3}}
    \put(2.5,0.5){\line(1,0){1.5}}
    \put(2.5,0.5){\vector(1,0){1}}
    \end{picture}
}
+
\raisebox{-0.3em}{
    \begin{picture}(4,1)
    \put(4,0.5){\line(-1,0){1.5}}
    \put(4,0.5){\vector(-1,0){1}}
    \put(2.5,0.5){\circle*{0.3}}
    \put(1.5,0.5){\circle*{0.3}}
    \put(1.5,0.5){\line(-1,0){1.5}}
    \put(1.5,0.5){\vector(-1,0){1}}
    \end{picture}
}.
\]

Note that univalent vertices
do not interact particularly well with crossings.
There is no relation to let you
pass a strand over or under a vertex,
and the orientation on a strand
may change when it goes through a crossing.

\section{The Hopf algebra $H$}

In this section,
we define a Hopf algebra $H$
consisting of formal linear combinations of certain diagrams.
The diagrams in $H$ are similar to those in $\TLdot$,
but with a special straight vertical edge called the {\em pole}.
No other strands are allowed to have endpoints
on the top or bottom of the diagram.
Strands are allowed to cross over or under the pole.
You can think of the pole as a kind of place holder.
In the next section we will replace it with
arbitrary numbers of parallel strands.

The turning, confetti, and cutting relations from $\TLdot$
still hold in $H$.
We also allow Reidemeister moves involving the pole.
That is,
we impose the relations
\[
\raisebox{-1.3em}{
    \begin{picture}(3,3)
    \put(0,1){\line(1,0){1.7}}
    \put(0,2){\line(1,0){1.7}}
    \put(2.3,1.5){\oval(1.4,1)[r]}
    \linethickness{2pt}
    \put(2,0){\line(0,1){3}}
    \end{picture}
}
=
\raisebox{-1.3em}{
    \begin{picture}(2,3)
    \put(0,1.5){\oval(2.5,1)[r]}
    \linethickness{2pt}
    \put(2,0){\line(0,1){3}}
    \end{picture}
}
=
\raisebox{-1.3em}{
    \begin{picture}(3,3)
    \put(0,1){\line(1,0){2}}
    \put(0,2){\line(1,0){2}}
    \put(2,1.5){\oval(2,1)[r]}
    \linethickness{2pt}
    \put(2,0){\line(0,1){0.7}}
    \put(2,1.3){\line(0,1){0.4}}
    \put(2,3){\line(0,-1){0.7}}
    \end{picture}
},
\qquad
\raisebox{-1.8em}{
    \begin{picture}(4,4)
    \qbezier(0,1)(1,1)(1.5,2)
    \qbezier(1.5,2)(2,3)(3,3)
    \put(4,3){\line(-1,0){1}}
    \qbezier(0,3)(1,3)(1.3,2.4)
    \qbezier(1.7,1.4)(2,1)(2.7,1)
    \put(4,1){\line(-1,0){0.7}}
    \linethickness{2pt}
    \put(3,0){\line(0,1){2.7}}
    \put(3,4){\line(0,-1){0.7}}
    \end{picture}
} =
\raisebox{-1.8em}{
    \begin{picture}(4,4)
    \put(0,1){\line(1,0){1}}
    \qbezier(1,1)(2,1)(2.5,2)
    \qbezier(2.5,2)(3,3)(4,3)
    \put(0,3){\line(1,0){0.7}}
    \qbezier(1.3,3)(2,3)(2.3,2.6)
    \qbezier(2.7,1.6)(3,1)(4,1)
    \linethickness{2pt}
    \put(1,0){\line(0,1){0.7}}
    \put(1,4){\line(0,-1){2.7}}
    \end{picture}
},
\]
and their horizontal reflections.
(Other versions of Reidemeister three
follow from Reidemeister two and the definition of a crossing.)

The product $\nabla \colon H \otimes H \to H$
is such that,
for diagrams $x$ and $y$,
$\nabla(x \otimes y)$ is obtained by stacking $x$ on top of $y$.
We write $xy$ for $\nabla(x \otimes y)$.

The unit $\eta \colon \mathbb{F} \to H$
is such that $\eta(1)$ is
the diagram that is empty except for the pole.

The tensor product $x \otimes y$
of two diagrams $x$ and $y$ in $H$
is obtained by placing $x$ to the left of $y$,
resulting in a diagram with two poles.
In general,
any diagram $z$ with two poles
represents an element of $H \otimes H$.
If $z$ contains strands that go from one pole to the other,
then use the cutting relation
to write $z$ as a sum of tensor products
of diagrams from $H$.

The coproduct $\Delta \colon H \to H \otimes H$
acts on any diagram
by splitting the pole into two parallel poles.
Every crossing where a strand passes over (or under) the pole
becomes a pair of crossings
where the strand passes over (or under) both poles.

The counit $\epsilon$
acts on any diagram $x$
by deleting the pole.
The result is a scalar multiple of the empty diagram in $\TLdot$,
and $\epsilon(x)$ is defined to be that scalar.

The antipode $S \colon H \to H$
acts on any diagram by a planar isotopy
that rotates the pole clockwise
through an angle of $180$ degrees.
Throughout the isotopy,
we must preserve the horizontal tangent vectors at every vertex.
The result is the same as rigidly rotating the diagram
and then multiplying the result by
$\sqrt{q}$ to the power of
the number of
inward oriented vertices
minus the number of
outward oriented vertices.

\begin{proposition}
$H$ satisfies the axioms of a Hopf algebra.
\end{proposition}

\begin{proof}
It is easy to check that $H$ satisfies
the axioms of a bialgebra.
It remains to check that the antipode satisfies:
$$\nabla \circ (S \otimes \id) \circ \Delta
= \eta \circ \epsilon
= \nabla \circ (\id \otimes S) \circ \Delta.
$$

Here is a schematic representation
of the effect of
$\nabla \circ (S \otimes \id) \circ \Delta$
on a diagram:
\[
\raisebox{-1.3em}{
    \begin{picture}(2,3)
    \put(0,1){\line(1,0){2}} \put(0,1){\line(0,1){1}}
    \put(2,2){\line(-1,0){2}} \put(2,2){\line(0,-1){1}}
    \linethickness{2pt}
    \put(1,0){\line(0,1){1}} \put(1,3){\line(0,-1){1}}
    \end{picture}
}
\xmapsto{\Delta}
\raisebox{-1.3em}{
    \begin{picture}(3,3)
    \put(0,1){\line(1,0){3}} \put(0,1){\line(0,1){1}}
    \put(3,2){\line(-1,0){3}} \put(3,2){\line(0,-1){1}}
    \linethickness{2pt}
    \put(1,0){\line(0,1){1}} \put(1,3){\line(0,-1){1}}
    \put(2,0){\line(0,1){1}} \put(2,3){\line(0,-1){1}}
    \end{picture}
}
\xmapsto{S \otimes \id}
\raisebox{-1.8em}{
    \begin{picture}(3,3)
    \put(3,0.5){\line(0,1){1}} \put(3,0.5){\line(-1,0){1}}
    \put(3,1.5){\line(-1,0){1}} \put(3,2.5){\line(0,1){1}}
    \put(3,2.5){\line(-1,0){1}} \put(3,3.5){\line(-1,0){1}}
    \put(2,2){\oval(2,1)[l]} \put(2,2){\oval(4,3)[l]}
    \linethickness{2pt}
    \put(2,0){\line(0,1){0.5}} \put(2,1.5){\line(0,1){0.3}}
    \put(2,2.5){\line(0,-1){0.3}} \put(2,3.5){\line(0,1){0.5}}
    \end{picture}
}
\xmapsto{\nabla}
\raisebox{-1.8em}{
    \begin{picture}(3,3)
    \put(3,0.5){\line(0,1){1}} \put(3,0.5){\line(-1,0){1}}
    \put(3,1.5){\line(-1,0){1}} \put(3,2.5){\line(0,1){1}}
    \put(3,2.5){\line(-1,0){1}} \put(3,3.5){\line(-1,0){1}}
    \put(2,2){\oval(2,1)[l]} \put(2,2){\oval(4,3)[l]}
    \linethickness{2pt}
    \put(2,0){\line(0,1){0.5}} \put(2,1.5){\line(0,1){1}}
    \put(2,3.5){\line(0,1){0.5}}
    \end{picture}
}.\]
First $\Delta$ splits the pole into two.
Then $S \otimes \id$
rotates the left pole clockwise,
bringing it above the other pole.
As usual,
the vertices do not rotate throughout this isotopy.
Although the resulting diagram is oddly shaped
and has one pole on top of the other,
it still represents an element of $H \otimes H$
by the same construction
as when the poles are side by side.
Finally,
$\nabla$ joins the two poles
so that the tensor product becomes multiplication by stacking.

Consider the last of the above sequence of four diagrams.
The curved part of the rectangle
represents a collection of parallel strands
that can be moved off the pole, one by one,
using Reidemeister two.
Thus the entire collection of strands
can be slid off the pole to the right.
We can then use an isotopy
to straighten out the rectangle again.
The overall effect
is to delete the pole from the original diagram
and insert a new pole some distance to the left.
But this exactly describes the action of $\eta \circ \epsilon$.
Thus
$$\nabla \circ (S \otimes \id) \circ \Delta
= \eta \circ \epsilon.$$

An upside-down version of this argument works for
$\nabla \circ (\id \otimes S) \circ \Delta$.
\end{proof}

\section{Representations of $H$}

If $h$ is a diagram in $H$,
let $\rho(h)$ be the diagram in $\TLdot$
given by replacing the pole in $h$ with an ordinary strand.
Extend $\rho$ by linearity
to an algebra morphism from $H$
to the algebra of automorphisms of the object $1$ in $\TLdot$.
Note that $\rho$ is a two-dimensional representation of $H$,
since its codomain is isomorphic to the algebra of
two-by-two matrices over $\mathbb{F}$.

Using the coproduct on $H$,
if $h$ is a diagram in $H$
then $\rho^{\otimes n}(h)$ is the diagram in $\TLdot$
given by threading $n$ parallel strands in place of the pole.

The most important relationship between $H$ and $\TL$
is that their actions  ``intertwine'' as follows.

\begin{theorem}
\label{thm:intertwine}
Suppose $h \in H$
and $f$ is a morphism in $\TL$ from $n$ to $m$.
Then
$$\rho^{\otimes m}(h) \circ f = f \circ \rho^{\otimes n}(h)$$
in the category $\TLdot$.
\end{theorem}

\begin{proof}
We can assume $h$ and $f$ are diagrams.
To obtain $\rho^{\otimes m}(h) \circ f$,
replace the pole in $h$ with $m$ parallel strands
and attach $f$ to the bottom.
To obtain $f \circ \rho^{\otimes n}(h)$,
replace the pole in $h$ with $n$ parallel strands
and attach $f$ to the top.
The resulting diagrams represent the same element of $\TLdot$,
since we can use Reidemeister moves to slide $f$ through $h$.
\end{proof}

\section{Generators and relations in $H$} 

We define the following elements of $H$.
\[
\begin{aligned}
&e = \raisebox{-1.3em}{
    \begin{picture}(3,3)
    \put(0,1.5){\line(1,0){1.2}}
    \put(3,1.5){\line(-1,0){1.2}}
    \put(0.4,1.5){\vector(-1,0){0}}
    \put(2.6,1.5){\vector(1,0){0}}
    \put(0,1.5){\circle*{0.3}}
    \put(3,1.5){\circle*{0.3}}
    \linethickness{2pt}
    \put(1.5,0){\line(0,1){3}}
    \end{picture}
},
\quad
&e_0 = \raisebox{-1.3em}{
    \begin{picture}(3,3)
    \put(0,1.5){\line(1,0){1.2}}
    \put(3,1.5){\line(-1,0){1.2}}
    \put(0.9,1.5){\vector(1,0){0}}
    \put(2.1,1.5){\vector(-1,0){0}}
    \put(0,1.5){\circle*{0.3}}
    \put(3,1.5){\circle*{0.3}}
    \linethickness{2pt}
    \put(1.5,0){\line(0,1){3}}
    \end{picture}
},
\quad
&k = \raisebox{-1.3em}{
    \begin{picture}(3,3)
    \put(0,1.5){\line(1,0){1.2}}
    \put(3,1.5){\line(-1,0){1.2}}
    \put(0.9,1.5){\vector(1,0){0}}
    \put(2.6,1.5){\vector(1,0){0}}
    \put(0,1.5){\circle*{0.3}}
    \put(3,1.5){\circle*{0.3}}
    \linethickness{2pt}
    \put(1.5,0){\line(0,1){3}}
    \end{picture}
},
\quad
&k' = \raisebox{-1.3em}{
    \begin{picture}(3,3)
    \put(0,1.5){\line(1,0){1.2}}
    \put(3,1.5){\line(-1,0){1.2}}
    \put(0.4,1.5){\vector(-1,0){0}}
    \put(2.1,1.5){\vector(-1,0){0}}
    \put(0,1.5){\circle*{0.3}}
    \put(3,1.5){\circle*{0.3}}
    \linethickness{2pt}
    \put(1.5,0){\line(0,1){3}}
    \end{picture}
},
\\
&f = \raisebox{-1.3em}{
    \begin{picture}(3,3)
    \put(0,1.5){\line(1,0){3}}
    \put(0.9,1.5){\vector(1,0){0}}
    \put(2.1,1.5){\vector(-1,0){0}}
    \put(0,1.5){\circle*{0.3}}
    \put(3,1.5){\circle*{0.3}}
    \linethickness{2pt}
    \put(1.5,0){\line(0,1){1.2}}
    \put(1.5,3){\line(0,-1){1.2}}
    \end{picture}
},
\quad
&f_0 = \raisebox{-1.3em}{
    \begin{picture}(3,3)
    \put(0,1.5){\line(1,0){3}}
    \put(0.4,1.5){\vector(-1,0){0}}
    \put(2.6,1.5){\vector(1,0){0}}
    \put(0,1.5){\circle*{0.3}}
    \put(3,1.5){\circle*{0.3}}
    \linethickness{2pt}
    \put(1.5,0){\line(0,1){1.2}}
    \put(1.5,3){\line(0,-1){1.2}}
    \end{picture}
},
\quad
&\ell = \raisebox{-1.3em}{
    \begin{picture}(3,3)
    \put(0,1.5){\line(1,0){3}}
    \put(0.4,1.5){\vector(-1,0){0}}
    \put(2.1,1.5){\vector(-1,0){0}}
    \put(0,1.5){\circle*{0.3}}
    \put(3,1.5){\circle*{0.3}}
    \linethickness{2pt}
    \put(1.5,0){\line(0,1){1.2}}
    \put(1.5,3){\line(0,-1){1.2}}
    \end{picture}
},
\quad
&\ell' = \raisebox{-1.3em}{
    \begin{picture}(3,3)
    \put(0,1.5){\line(1,0){3}}
    \put(0.9,1.5){\vector(1,0){0}}
    \put(2.6,1.5){\vector(1,0){0}}
    \put(0,1.5){\circle*{0.3}}
    \put(3,1.5){\circle*{0.3}}
    \linethickness{2pt}
    \put(1.5,0){\line(0,1){1.2}}
    \put(1.5,3){\line(0,-1){1.2}}
    \end{picture}
}.
\end{aligned}
\]

\begin{lemma} \label{lem:Hgenerators}
$H$ is generated by the above eight elements.
\end{lemma}

\begin{proof}
Start with an arbitrary diagram in $H$.
Apply the definition of a crossing
to eliminate any crossings that do not involve the pole.
Use the cutting relation
to cut all strands into segments
that cross the pole at most once.
Use the turning relations
to straighten out all of the strands.
Finally,
use the confetti relations
to eliminate any strands that do not cross the pole.
We are left with only
horizontal segments that cross the pole exactly once.
The eight generators consist of every pair of orientations
for either type of crossing.
\end{proof}

We now give the Hopf algebra structure of $H$.
To save space,
we only list the four generators
in which the strand passes under the pole.
These calculations remain the same 
if we switch the crossing.

\begin{lemma} \label{lem:Hhopf}
In $H$,
the coproduct satisfies:
$$
\begin{aligned}
\Delta(e) = e \otimes k + k' \otimes e,
\quad
&
\Delta(e_0) = e_0 \otimes k' + k \otimes e_0,
\\
\Delta(k) = k \otimes k + e_0 \otimes e,
\quad
&
\Delta(k') = k' \otimes k' + e \otimes e_0,
\end{aligned}
$$
the counit satisfies:
$$
\epsilon(e) = \epsilon(e_0) = 0,
\quad 
\epsilon(k) = \epsilon(k') = 1,
$$
and the antipode satisfies:
$$S(e) = qe,
\quad
S(e_0) = q^{-1}e_0,
\quad
S(k) = k',
\quad
S(k') = k.
$$
\end{lemma}

\begin{proof}
These follow immediately from the definitions.
\end{proof}

We list some relations satisfied by the generators of $H$.
We do not attempt a complete presentation of $H$,
since we will soon be taking a quotient anyway.

\begin{lemma} \label{lem:Hrelations}
$H$ satisfies the relations
\begin{itemize}
\item $k'k + q^{-1}ee_0 = 1$,
\item $kk' + qe_0e = 1$,
\item $ek' + qk'e = 0$,
\item $ef - fe = (q - q^{-1})(\ell k - k' \ell')$.
\end{itemize}
\end{lemma}

\begin{proof}
The first three relations follow from Reidemeister two:
\[
\raisebox{-1.3em}{
    \begin{picture}(3,3)
    \put(0,2){\line(1,0){1.7}}
    \put(0,1){\line(1,0){1.7}}
    \put(0,2){\circle*{0.3}}
    \put(0,1){\circle*{0.3}}
    \put(1.7,2){\vector(-1,0){1.1}}
    \put(0,1){\vector(1,0){1.0}}
    \put(2.3,1.5){\oval(1.4,1)[r]}
    \linethickness{2pt}
    \put(2,0){\line(0,1){3}}
    \end{picture}
} = \sqrt{q},
\quad
\raisebox{-1.3em}{
    \begin{picture}(3,3)
    \put(0,2){\line(1,0){1.7}}
    \put(0,1){\line(1,0){1.7}}
    \put(0,2){\circle*{0.3}}
    \put(0,1){\circle*{0.3}}
    \put(0,2){\vector(1,0){1.0}}
    \put(1.7,1){\vector(-1,0){1.1}}
    \put(2.3,1.5){\oval(1.4,1)[r]}
    \linethickness{2pt}
    \put(2,0){\line(0,1){3}}
    \end{picture}
} = \frac{1}{\sqrt{q}},
\quad
\raisebox{-1.3em}{
    \begin{picture}(3,3)
    \put(0,2){\line(1,0){1.7}}
    \put(0,1){\line(1,0){1.7}}
    \put(0,2){\circle*{0.3}}
    \put(0,1){\circle*{0.3}}
    \put(1.7,2){\vector(-1,0){1.1}}
    \put(1.7,1){\vector(-1,0){1.1}}
    \put(2.3,1.5){\oval(1.4,1)[r]}
    \linethickness{2pt}
    \put(2,0){\line(0,1){3}}
    \end{picture}
} = 0.
\]
The fourth relation follows from Reidemeister three:
\[
\raisebox{-1.8em}{
    \begin{picture}(5,4)
    \put(0,1){\circle*{0.3}}
    \put(0,1){\line(1,0){0.5}}
    \put(0.9,1){\vector(1,0){0}}
    \qbezier(0.5,1)(1.5,1)(2,2)
    \qbezier(2,2)(2.5,3)(3.5,3)
    \put(5,3){\line(-1,0){1.5}}
    \put(5,3){\vector(-1,0){0.9}}
    \put(5,3){\circle*{0.3}}
    \put(0,3){\circle*{0.3}}
    \put(0,3){\line(1,0){0.5}}
    \put(0.5,3){\vector(-1,0){0}}
    \qbezier(0.5,3)(1.5,3)(1.8,2.4)
    \qbezier(2.2,1.4)(2.5,1)(3.2,1)
    \put(5,1){\line(-1,0){1.2}}
    \put(4.6,1){\vector(1,0){0}}
    \put(5,1){\circle*{0.3}}
    \linethickness{2pt}
    \put(3.5,0){\line(0,1){2.7}}
    \put(3.5,4){\line(0,-1){0.7}}
    \end{picture}
} =
\raisebox{-1.8em}{
    \begin{picture}(5,4)
    \put(0,1){\circle*{0.3}}
    \put(0,1){\line(1,0){1.5}}
    \put(0,1){\vector(1,0){0.9}}
    \qbezier(1.5,1)(2.5,1)(3,2)
    \qbezier(3,2)(3.5,3)(4.5,3)
    \put(5,3){\line(-1,0){0.5}}
    \put(4.1,3){\vector(-1,0){0}}
    \put(5,3){\circle*{0.3}}
    \put(0,3){\circle*{0.3}}
    \put(0,3){\line(1,0){1.2}}
    \put(0.4,3){\vector(-1,0){0}}
    \qbezier(1.8,3)(2.5,3)(2.8,2.6)
    \qbezier(3.2,1.6)(3.5,1)(4.5,1)
    \put(5,1){\line(-1,0){0.5}}
    \put(4.5,1){\vector(1,0){0}}
    \put(5,1){\circle*{0.3}}
    \linethickness{2pt}
    \put(1.5,0){\line(0,1){0.7}}
    \put(1.5,4){\line(0,-1){2.7}}
    \end{picture}
}.\]
In each case,
we can express the equation of diagrams
in terms of the generators of $H$,
using the method described in the proof of Lemma \ref{lem:Hgenerators}.
After some algebraic manipulation,
we obtain the desired relations.
\end{proof}

\section{A quotient of $H$}

Let $H'$ be the quotient of $H$
by the intersection of the kernels of all $\rho^{\otimes n}$.
The aim of this section is to prove the following.

\begin{theorem} \label{thm:presentation}
$H'$ has generators $e$, $f$, $k$ and $k^{-1}$,
which satisfy the relations:
$$kk^{-1} = k^{-1}k = 1$$
and
$$
ek = -q^{-1}ke,
\quad fk = -qkf,
\quad ef - fe = (q-q^{-1})(k^2 - k^{-2}).
$$
The Hopf algebra structure on $H'$ is given by the coproduct:
$$
\Delta(e) = e \otimes k + k^{-1} \otimes e,
\quad \Delta(f) = f \otimes k + k^{-1} \otimes f,
\quad \Delta(k^{\pm 1}) = k^{\pm 1} \otimes k^{\pm 1},
$$
the counit:
$$
\epsilon(e) = \epsilon(f) = 0,
\quad \epsilon(k^{\pm 1}) = 1,
$$
and the antipode:
$$
S(e) = qe,
\quad
S(f) = q^{-1}f,
\quad
S(k^{\pm 1}) = k^{\mp 1}.
$$
\end{theorem}

To save us some work,
we use the following symmetry of $H$.

\begin{definition}
The {\em Cartan involution} of $H$
is the linear map $\theta \colon H \to H$
that acts on a diagram
by rotating it $180$ degrees around the pole,
and reversing the direction of all arrows.
\end{definition}

Thus $\theta$ permutes the generators of $H$ as follows.
$$e \leftrightarrow f, \quad
  e_0 \leftrightarrow f_0, \quad
  k \leftrightarrow \ell', \quad
  k' \leftrightarrow \ell.$$

\begin{lemma} \label{lem:cartan}
The Cartan involution is a bialgebra automorphism of $H$,
preserves the kernel of $\rho^{\otimes n}$ for all $n$,
and satisfies $\theta \circ S = S^{-1} \circ \theta$.
\end{lemma}

\begin{proof}
The bialgebra operations on $H$
have diagrammatic descriptions
that commute with $\theta$.
We can also say $\rho^{\otimes n}$ commutes with $H$,
if we interpret $\theta$ as acting on $\TLdot$ in the obvious way.
Finally,
$\theta$ does not commute with $S$,
but instead reverses the direction of rotation of the pole
in the definition of $S$.
\end{proof}

\begin{lemma} \label{lem:kernel}
For all $n \ge 0$, we have:
\[
\rho^{\otimes n} (e_0) = \rho^{\otimes n} (f_0) = 0,
\quad \rho^{\otimes n} (k) = \rho^{\otimes n} (\ell),
\quad \rho^{\otimes n} (k') = \rho^{\otimes n} (\ell').
\]
\end{lemma}

\begin{proof}
The proof is by induction on $n$.
The case $n = 0$ is easy.
The case $n = 1$ is a simple computation
involving diagrams with a single crossing.
For $n > 1$,
use the formulae for the coproduct
taken from Lemma \ref{lem:Hhopf}
and Lemma \ref{lem:cartan}.
\end{proof}

To prove Theorem \ref{thm:presentation},
let $k^{-1} = k'$ and combine Lemmas
\ref{lem:Hhopf},
\ref{lem:Hrelations},
\ref{lem:cartan}
and \ref{lem:kernel}.

\section{Connection to $\Uq$}

In this section we make the connection between $\Uq$ and $H'$.

We use the definition of $\Uq$
given in \cite{kassel} and \cite{klimykschmudgen}.
The most interesting relations are:
\[
\begin{aligned}
&
EK = q^{-2}KE, \quad FK = q^2 KF,
  \quad EF - FE = (K-K^{-1})/(q-q^{-1}),
\\
&
\Delta(E) = 1 \otimes E + E \otimes K, 
  \quad \Delta(F) = K^{-1} \otimes F + F \otimes 1,
\\
&
S(E) = -EK^{-1}, \quad S(F) = -KF.
\end{aligned}
\]
The other relations are
the definition of $\epsilon$
and some obvious relations involving only $K^{\pm 1}$.

Define $\phi \colon \Uq \to H'$ by
$$
\phi(K^{\pm 1}) = k^{\pm 2}, \quad
\phi(E) = \frac{1}{q-q^{-1}} ek, \quad
\phi(F) = \frac{1}{q-q^{-1}} k^{-1}f.
$$

\begin{theorem}
\label{thm:end}
The above $\phi$ is a well-defined morphism of Hopf algebras.
The image of $\phi$ is the algebra of words of even length
in the generators $e$, $f$ and $k^{\pm 1}$.
The representation $\rho \circ \phi$ of $\Uq$
is isomorphic to $V_{-1,1}$.
The kernel of $\phi$ is
the intersection of the kernels of the representations
$V_{-1,1}^{\otimes n}$.
\end{theorem}

\begin{proof}
To check that $\phi$ is well defined,
simply check all of the defining relations of $\Uq$.
To see that the image of $\phi$ is as claimed,
note that it is easy to convert any word of even length
to a power of $-q$ times a word in the image of $\phi$.

We can compute $\rho \circ \phi$ completely,
or just enough to identify it by a process of elimination.
We know that $\rho \circ \phi$ is a two-dimensional representation,
so it is either trivial, $V_{1,1}$, or $V_{-1,1}$.
But
$$\rho \circ \phi(E) \neq 0,$$
so it is not trivial.
Also,
$$\rho \circ \phi(KE) = (-q) \rho \circ \phi(E)$$
so $K$ has an eigenvalue $-q$,
and the representation must be $V_{-1,1}$.

The statement about the kernel of $\phi$
follows immediately from the definition of $H'$.
\end{proof}

\begin{corollary}
If $\mathbb{F} = \mathbb{C}$ and $q$ is not a root of unity
then $\Uq$ is isomorphic to the algebra of words of even length 
in the generators $e$, $f$ and $k^{\pm 1}$ of $H'$.
\end{corollary}

\begin{proof}
We must show that $\phi$ is injective.
We use basic properties of representations of $\Uq$
when $q$ is not a root of unity.

By \cite[Theorem 7.13]{klimykschmudgen},
the intersection of the kernels of the representations $V_{1,n}$.
is trivial.
But $V_{1,n}$ is a summand of $V_{1,1}^{\otimes n}$,
so the intersection of the kernels of tensor powers of $V_{1,1}$
is trivial.
There is an isomorphism from $\Uq$ to $\Unegq$
that switches $V_{-1,1}$ and $V_{1,1}$,
so the intersection of the kernels of tensor powers of $V_{-1,1}$
is also trivial.
Thus the kernel of $\phi$ is trivial.
\end{proof}

Even if $q$ is a root of unity,
Theorem \ref{thm:intertwine} shows that
$\TL$ is a category of
all tensor powers of the representation $V_{-1,1}$ of $\Uq$,
and some of the morphisms between them.
If $q$ is not a root of unity then
$\TL$ includes {\em all} such morphisms,
but I do not know a diagrammatic proof of this fact.


\begin{thebibliography}{10}

\bibitem{BS}
J.~Baez and M.~Stay.
\newblock Physics, topology, logic and computation: a {R}osetta {S}tone.
\newblock In {\em New structures for physics}, volume 813 of {\em Lecture Notes
  in Phys.}, pages 95--172. Springer, Heidelberg, 2011.

\bibitem{FK}
Igor~B. Frenkel and Mikhail~G. Khovanov.
\newblock Canonical bases in tensor products and graphical calculus for
  {$U_q(\mathfrak{sl}_2)$}.
\newblock {\em Duke Math. J.}, 87(3):409--480, 1997.

\bibitem{jones}
Vaughan F.~R. Jones.
\newblock In and around the origin of quantum groups.
\newblock In {\em Prospects in mathematical physics}, volume 437 of {\em
  Contemp. Math.}, pages 101--126. Amer. Math. Soc., Providence, RI, 2007.

\bibitem{kashiwara}
M.~Kashiwara.
\newblock On crystal bases of the {$Q$}-analogue of universal enveloping
  algebras.
\newblock {\em Duke Math. J.}, 63(2):465--516, 1991.

\bibitem{kassel}
Christian Kassel.
\newblock {\em Quantum groups}, volume 155 of {\em Graduate Texts in
  Mathematics}.
\newblock Springer-Verlag, New York, 1995.

\bibitem{kauffman}
Louis~H. Kauffman.
\newblock An invariant of regular isotopy.
\newblock {\em Trans. Amer. Math. Soc.}, 318(2):417--471, 1990.

\bibitem{klimykschmudgen}
Anatoli Klimyk and Konrad Schm{\"u}dgen.
\newblock {\em Quantum groups and their representations}.
\newblock Texts and Monographs in Physics. Springer-Verlag, Berlin, 1997.

\bibitem{lauda}
Aaron~D. Lauda.
\newblock A categorification of quantum {${\rm sl}(2)$}.
\newblock {\em Adv. Math.}, 225(6):3327--3424, 2010.

\bibitem{lieb}
E.~H. {Lieb}.
\newblock {Residual Entropy of Square Ice}.
\newblock {\em Physical Review}, 162:162--172, October 1967.

\bibitem{morton}
Hugh~R. Morton.
\newblock Skein theory and the {M}urphy operators.
\newblock {\em J. Knot Theory Ramifications}, 11(4):475--492, 2002.
\newblock Knots 2000 Korea, Vol. 2 (Yongpyong).

\end{thebibliography}
\end{document}